\documentclass{amsart}
\usepackage{graphicx}
\usepackage{enumitem}
\newtheorem{theorem}{Theorem}[section]
\newtheorem{corollary}[theorem]{Corollary}
\newtheorem{lemma}[theorem]{Lemma}

\theoremstyle{definition}

\theoremstyle{remark}

\numberwithin{equation}{section}

\newcommand{\set}[1]{\left\{#1\right\}}

\newcommand{\br}[1]{\left(#1\right)}
\newcommand{\mat}[1]{\left[#1\right]}

\begin{document}

\title[On the Integral Representations of the $k$-Pell and  $k$-Pell-Lucas Numbers]{On the Integral Representations of the $k$-Pell and  $k$-Pell-Lucas Numbers}%
\author{ACHARIYA NILSRAKOO}%
\address{Department of Mathematics, Faculty of Science,\\
Ubon Ratchathani Rajabhat University,
Ubon Ratchathani, 34000}%
 \email{achariya.n@ubru.ac.th}
\author{Weerayuth Nilsrakoo}%
\address{Department of Mathematics, Faculty of Science,\\
 Ubon Ratchathani University,
 Ubon Ratchathani, 34190\\
THAILAND}
 \email{weerayuth.ni@ubu.ac.th}
 \date{}
\thanks{Mathematics Subject Classification (2020): 11B39, 11B37.}
\thanks{Key words and phrases: $k$-Pell number; $k$-Pell-Lucas number; integral representation}
\thanks{The corresponding author is Weerayuth Nilsrakoo}%
%
\begin{abstract}In this paper, the integral representations of the $k$-Pell and $k$-Pell-Lucas numbers are presented. By employing Binet’s formulas for these numbers, we derive several identities and verify their integral representations using basic integral calculus. Our results are also deduced and related to the Fibonacci, Lucas, Pell, and  Pell-Lucas numbers.
\end{abstract}
\maketitle

\section{Introduction}\label{Introduction}

Fibonacci, Lucas, Pell, and Pell–Lucas sequences have been discussed in many articles and books (see \cite{AndBag2020,Bic1915,Dil2000,ErdKes2022,GlaZho2015,Hor1971,HorMah1985,Kos2018,ANil2024,
Ste2022,Ste2023,Vaj2008}). Recall that the Fibonacci numbers $F_n$ are defined via the recurrence relation
\begin{equation*}
F_n=F_{n-1}+F_{n-2}
\end{equation*}
for $n\geq 2$ with $F_0=0$ and $F_1=1$. The Lucas numbers $L_n$ are defined via the recurrence relation
\begin{equation*}
L_n=L_{n-1}+L_{n-2}
\end{equation*}
for $n\geq 2$ with $L_0=2$ and $L_1=1$.

There are several ways to represent these numbers, one of which is the integral representation. The first example of integral representations for Fibonacci numbers of even orders, using an approach based on the (Gaussian) hypergeometric function, was presented in the paper by Dilcher \cite{Dil2000}  from 2000, where he showed that
\begin{equation*}
F_{2n}=\frac{n}{2}\br{\frac{3}{2}}^{n-1}\int_{0}^{\pi}\br{1+\frac{\sqrt{5}}{3}\cos x}^{n-1}\sin xdx.
\end{equation*}

In 2015, Glasser and Zhou \cite{GlaZho2015} worked out an explicit integral representation for $F_n$ involving trigonometric functions. Indeed, the main result of their paper is a representation of the form
\begin{equation*}
F_{n} = \frac{\alpha^n}{\sqrt{5}}-\dfrac{2}{\pi} \int_{0}^{\infty}\br{\dfrac{\sin(x/2)}{x}}\br{\dfrac{\cos(2nx)-2\sin(nx)\sin x}{5\sin^2x+\cos^2 x}}dx,
\end{equation*}
where $\alpha=\frac{1+\sqrt{5}}{2}$ is the golden ratio. Another representation is given
by Andrica and Bagdasar in \cite{AndBag2020}. In a recent year, Stewart \cite{Ste2022} derived some appealing integral representations for Fibonacci and Lucas numbers. For instance, he proved the representations
\begin{equation*}F_{\ell n}=\dfrac{nF_\ell}{2^n}\int_{-1}^{1} (L_\ell+\sqrt{5}F_\ell x)^{n-1} dx
\end{equation*}
and
\begin{equation*}L_{\ell n}=\frac{1}{2^n}\int_{-1}^{1} (L_\ell+\sqrt{5}(n+1)F_\ell x)(L_\ell+\sqrt{5}F_\ell x)^{n-1} dx,
\end{equation*}
where $\ell$ and $n$ are  non-negative integers. The special case of this identity for $\ell=1$ is also discussed in Stewart’s paper \cite{Ste2023}  from 2023.

Like Fibonacci and Lucas numbers, the Pell family is widely used. Pell and Pell–Lucas numbers also provide boundless opportunities to experiment, explore, and conjecture.
Recall that Pell number $P_n$ is defined by the recurrence relation
\begin{equation*}
P_n=2P_{n-1}+P_{n-2}
\end{equation*}
for $n\geq 2$ with $P_0=0$ and $P_1=1$.  The Pell-Lucas number $Q_n$ is defined by the recurrence relation
   \begin{equation*}
Q_n=2Q_{n-1}+Q_{n-2}
\end{equation*}
for $n\geq 2$ with $Q_0=2$ and $Q_1=2$.  The Binet’s formulas for
the Pell and  Pell-Lucas numbers are
\begin{equation*}P_n=\frac{1}{2\sqrt{2}}\br{\varphi^n-\frac{(-1)^n}{\varphi^n}}
\,\,\text{and}\,\,Q_n=\varphi^n+\frac{(-1)^n}{\varphi^n},
\end{equation*}
where $\varphi=1+\sqrt{2}$. The integral representations for Pell and Pell-Lucas numbers are studied by the first author in \cite{ANil2024} as follows:
\begin{equation*}
P_{\ell n}=\frac{nP_{\ell}}{2^n}\int_{-1}^{1}(Q_{\ell}+2\sqrt{2} P_{\ell} x)^{n-1}dx
\end{equation*}
and
\begin{equation*}
Q_{\ell n}=\frac{1}{2^n}\int_{-1}^{1}(Q_{\ell}+2\sqrt{2}(n+1) P_{\ell} x)(Q_{\ell}+2\sqrt{2} P_{\ell} x)^{n-1}dx,
\end{equation*}
where $\ell$ and $n$ are  non-negative integers.

The generalization of Pell and Pell-Lucas numbers was found by studying the recursion and introduced by Catarino \cite{Cat2013}, and  Catarino and Vasco \cite{CatVas2013}  in 2013; see more details in Section \ref{Preliminaries}. According to Stewart \cite{Ste2022} and Nilsrakoo \cite{ANil2024}, we give new integral representations  of the $k$-Pell  and $k$-Pell-Lucas numbers by using Binet’s formulas to establish some identities and simple integral calculus to prove them.
\section{Preliminaries}\label{Preliminaries}
In this section, we briefly recall some concepts and results that are required for the proofs of the main results.  Let $k$ and $n$ be non-negative integers with $k\neq 0$. In 2013, Catarino \cite{Cat2013} introduced and studied a generalization of Pell numbers as follows: the $k$-Pell number $P_{k,n}$ is defined by the recurrence relation
\begin{equation}\label{k- Pell-eq}
P_{k,n}=2P_{k,n-1}+kP_{k,n-2}
\end{equation}
for $n\geq 2$ with $P_{k,0}=0$ and $P_{k,1}=1$.	Subsequently, Catarino and Vasco \cite{CatVas2013} introduced and studied a generalization of Pell-Lucas numbers as follows: the $k$-Pell-Lucas number $Q_{k,n}$ is defined by the recurrence relation
\begin{equation}\label{k- Pell_Lucas-eq}
Q_{k,n}=2Q_{k,n-1}+kQ_{k,n-2},
\end{equation}
for $n\geq 2$ with $Q_{k,0}=2$ and $Q_{k,1}=2$. The tables presented below contain initial terms of the
sequences $\set{P_{k,n}}$ and $\set{Q_{k,n}}$ for selected values of
$k$ (Tables \ref{table-cpkn} and \ref{table-cqkn}).
\begin{table}[h]
\caption{Initial terms of the
$k$-Pell numbers $\set{P_{k,n}}$}\label{table-cpkn}
	\centering
	\begin{tabular}{lrrrrrrrrrrr}
		\hline\noalign{\smallskip}
		$n$&0&1 & 2&3&4&5&6&7&8&9&10\\\noalign{\smallskip}\hline\hline\noalign{\smallskip}
$P_{1,n}$&0&	1&	2&	5&	12&	29&	70&	169&	408&	985&	2,378\\
$P_{2,n}$&0&	1&	2&	6&	16&	44&	120&	328&	896&	2,448&	6,688
\\
$P_{3,n}$&0&	1&	2&	7&	20&	61&	182&	547&	1,640&	4,921&	14,762
\\
$P_{4,n}$&0&	1&	2&	8&	24&	80&	256&	832&	2,688&	8,704&	28,160
\\
$P_{5,n}$&0&	1&	2&	9&	28&	101&	342&	1,189&	4,088&	14,121&	48,682
\\
$P_{6,n}$&0&	1&	2&	10&	32&	124&	440&	1,624&	5,888&	21,520&	78,368 \\
	\noalign{\smallskip}	\hline
	\end{tabular}
\end{table}
\begin{table}[h]
\caption{Initial terms of the
$k$-Pell-Lucas numbers $\set{Q_{k,n}}$}\label{table-cqkn}
	\centering
	\begin{tabular}{lrrrrrrrrrrr}
		\hline\noalign{\smallskip}
		$n$&0&1 & 2&3&4&5&6&7&8&9&10\\\noalign{\smallskip}\hline\hline\noalign{\smallskip}
$Q_{1,n}$&2&	2&	6&	14&	34&	82&	198&	478&	1,154&	2,786&	6,726
\\
$Q_{2,n}$&2&	2&	8&	20&	56&	152&	416&	1,136&	3,104&	8,480&	23,168
\\
$Q_{3,n}$&2&	2&	10&	26&	82&	242&	730&	2,186&	6,562&	19,682&	59,050
\\
$Q_{4,n}$&2&	2&	12&	32&	112&	352&	1,152&	3,712&	12,032&	38,912&	125,952
\\
$Q_{5,n}$&2&	2&	14&	38&	146&	482&	1,694&	5,798&	20,066&	69,122&	238,574
\\
$Q_{6,n}$&2&	2&	16&	44&	184&	632&	2,368&	8,528&	31,264&	113,696&	414,976\\
	\noalign{\smallskip}	\hline
	\end{tabular}
\end{table}

As we can see, for $k=1$, the classical Pell and classical Pell–Lucas numbers are obtained. If $k=4$, the classical Fibonacci  and classical Lucas numbers appear as follows: $F_n=\frac{1}{2^{n-1}}P_{4,n}$ and $L_n=\frac{1}{2^{n}}Q_{4,n}$. Moreover, sequences
$\set{P_{2,n}}$, $\set{P_{3,n}}$, $\set{P_{4,n}}$, and  $\set{P_{5,n}}$ are listed in The Online Encyclopaedia of Integer Sequences \cite{OEIS2024} under the symbols A002605, A015518,
A085449, and A002532, respectively, while sequences $\set{Q_{2,n}}$, $\set{Q_{3,n}}$, $\set{Q_{4,n}}$, and  $\set{Q_{6,n}}$ under the symbols A080040,  A102345, A087131, and  A127226, respectively.

The recurrence relations \eqref{k- Pell-eq} and \eqref{k- Pell_Lucas-eq} generate characteristic
equations of the form
\begin{equation*}
r^2-2r-k=0.
\end{equation*}
Since $k\geq 1$, this equation has two roots, $r_1=1+\sqrt{1+k}$ and $r_2=1-\sqrt{1+k}$. Therefore, the Binet’s formulas  for the $k$-Pell numbers $\set{P_{k,n}}$ and the $k$-Pell-Lucas numbers $\set{Q_{k,n}}$ are
\begin{equation}\label{k-Pell-Binet-eq}
P_{k,n}=\frac{1}{2\sqrt{1+k}}\br{\varphi_k^n-\frac{(-k)^n}{\varphi_k^n}}
\end{equation}
and
\begin{equation}\label{k-Pell-Lucas-Binet-eq}
Q_{k,n}=\varphi_k^n+\frac{(-k)^n}{\varphi_k^n},
\end{equation}
where $\varphi_k=1+\sqrt{1+k}$; see also \cite[Proposition 1]{Cat2013} and \cite[Proposition 1]{CatVas2013}.

The following lemmas are our main tools for proving the integral representations of  the $k$-Pell and  $k$-Pell-Lucas numbers.
\begin{lemma}\label{lemma2.1}Let $k$ and $n$ be non-negative integers with $k\neq 0$. Then the following hold:
\begin{enumerate}[label=(\roman*)]
       \item  \label{lemma2.1-1} $Q_{k,n}+2\sqrt{1+k}\,P_{k,n}=2\varphi_k^n;$
       \item  \label{lemma2.1-2} $Q_{k,n}-2\sqrt{1+k}\,P_{k,n}=2\frac{(-k)^n}{\varphi_k^n};$
        \item  \label{lemma2.1-3} $Q_{k,n}^2-4(1+k)P_{k,n}^2=4(-k)^n$.
     \end{enumerate}
\end{lemma}
\begin{proof}\begin{enumerate}[label=(\roman*)]
                 \item   Combining Binet’s formulas \eqref{k-Pell-Lucas-Binet-eq} and \eqref{k-Pell-Binet-eq} gives
          \begin{align*}
          Q_{k,n}+2\sqrt{1+k}\,P_{k,n}
          =\br{\varphi_k^n+\frac{(-k)^n}{\varphi_k^n}}+\br{\varphi_k^n-\frac{(-k)^n}{\varphi_k^n}}  =2\varphi_k^n.
          \end{align*}
                \item  Subtracting Binet’s formulas \eqref{k-Pell-Lucas-Binet-eq} and \eqref{k-Pell-Binet-eq} gives
                 \begin{align*}
          Q_{k,n}-2\sqrt{1+k}\,P_{k,n}
          =\br{\varphi_k^n+\frac{(-k)^n}{\varphi_k^n}}-\br{\varphi_k^n-\frac{(-k)^n}{\varphi_k^n}}  =2\frac{(-k)^n}{\varphi_k^n}.
          \end{align*}
                 \item It follows from  \ref{lemma2.1-1} and \ref{lemma2.1-2} that
                      \begin{align*}
          Q_{k,n}^2-4(1+k)P_{k,n}^2&= Q_{k,n}^2-\br{2\sqrt{1+k}P_{k,n}}^2\\
          &=\br{Q_{k,n}+2\sqrt{1+k}\,P_{k,n}}\br{Q_{k,n}-2\sqrt{1+k}\,P_{k,n}}\\
&=\br{2\varphi_k^n}{2\frac{(-k)^n}{\varphi_k^n}}\\
&=4(-k)^n.
          \end{align*}
  \end{enumerate}
This completes the proof.
\end{proof}
\begin{lemma}\label{lemma2.2}Let $k$, $m$, and $n$ be non-negative integers with $k\neq 0$. Then the following hold:
\begin{enumerate}[label=(\roman*)]
       \item \label{lemma2.2-1} $2P_{k,m+n}=P_{k,m}Q_{k,n}+P_{k,n}Q_{k,m};$
       \item \label{lemma2.2-2} $2Q_{k,m+n}=Q_{k,m}Q_{k,n}+4(1+k)P_{k,m}P_{k,n}.$
     \end{enumerate}
\end{lemma}
\begin{proof}Using Binet’s formulas  \eqref{k-Pell-Binet-eq} and  \eqref{k-Pell-Lucas-Binet-eq},  we obtain
          \begin{align*}
         P_{k,m}Q_{k,n}+P_{k,n}Q_{k,m}
         &=\br{\frac{1}{2\sqrt{1+k}}\br{\varphi_k^m-\frac{(-k)^m}{\varphi_k^m}}}\br{\varphi_k^n
         +\frac{(-k)^n}{\varphi_k^n}}\\
        &\quad+\br{\frac{1}{2\sqrt{1+k}}\br{\varphi_k^n-\frac{(-k)^n}{\varphi_k^n}}}        \br{\varphi_k^m+\frac{(-k)^m}{\varphi_k^m}}\\
        &=\frac{1}{\sqrt{1+k}}\br{\varphi_k^{m+n}-\frac{(-k)^{m+n}}{\varphi_k^{m+n}}}\\
         &=2P_{k,m+n}
          \end{align*}
and
                      \begin{align*}
        Q_{k,m}&Q_{k,n}+4(1+k)P_{k,m}P_{k,n}\\&=\br{\varphi_k^m+\frac{(-k)^m}{\varphi_k^m}}\br{\varphi_k^n
         +\frac{(-k)^n}{\varphi_k^n}}\\
        &\quad+4(1+k)\br{\frac{1}{2\sqrt{1+k}}\br{\varphi_k^m-\frac{(-k)^m}{\varphi_k^m}}}\br{\frac{1}{2\sqrt{1+k}}\br{\varphi_k^n-\frac{(-k)^n}{\varphi_k^n}}}        \\
        &=2\br{\varphi_k^{m+n}+\frac{(-k)^{m+n}}{\varphi_k^{m+n}}}\\&=2Q_{k,m+n}.
          \end{align*}
Hence, \ref{lemma2.2-1} and  \ref{lemma2.2-2} complete the proof.
\end{proof}
\section{Main Results}
In this section, thanks to the technique of \cite{Ste2022},  we obtain new integral representations for the $k$-Pell and $k$-Pell-Lucas numbers. We start with the integral representation for the $k$-Pell number $P_{k,\ell n}$ which can be found by employing other known relations between the two numbers $P_{k, \ell}$ and $Q_{k,\ell}$.
\begin{theorem}\label{C-k-Pell-Theorem1}For $k$, $\ell$, and $n$ are non-negative integers with $k\neq 0$, the $k$-Pell numbers $P_{k,\ell n}$ can be represented by the integral
\begin{equation}\label{ckPell}
P_{k,\ell n}=\frac{nP_{k,\ell}}{2^n}\int_{-1}^{1}(Q_{k,\ell}+2\sqrt{1+k}\,P_{k,\ell} x)^{n-1}dx.
\end{equation}
\end{theorem}
\begin{proof}For $n=0$ or $\ell=0$, we have done. Let us assume that $\ell, n>0$, straightforward simple integration leads to
\begin{align}\label{int-pln-eq1}
 &\int_{-1}^{1}(Q_{k,\ell}+2\sqrt{1+k}\, P_{k,\ell} x)^{n-1}dx\nonumber\\
 &=\frac{1}{2n\sqrt{1+k}\,P_{k,\ell}}\mat{(Q_{k,\ell}+2\sqrt{1+k}\, P_{k,\ell} x)^{n}}_{-1}^{1}\nonumber\\ &=\frac{1}{2n\sqrt{1+k}\,P_{k,\ell}}\mat{(Q_{k,\ell}+2\sqrt{1+k} \,P_{k,\ell})^{n}-(Q_{k,\ell}-2\sqrt{1+k}\, P_{k,\ell})^n}.
\end{align}
Applying \ref{lemma2.1-1} and \ref{lemma2.1-2} of Lemma \ref{lemma2.1}  in  \eqref{int-pln-eq1} with $n$ replaced with $\ell$, we get
\begin{align*}
 \int_{-1}^{1}(Q_{k,\ell}+2\sqrt{1+k}\, P_{k,\ell} x)^{n-1}dx&=\frac{1}{2n\sqrt{1+k}\,P_{k,\ell}}\mat{\br{2\varphi_k^\ell}^{n}-\br{2\frac{(-k)^\ell}{\varphi_k^\ell}}^n}\,\nonumber\\
 &=\frac{2^n}{nP_{k,\ell}}\mat{\frac{1}{2\sqrt{1+k}}\br{\varphi_k^{\ell n}
 -\frac{(-k)^{\ell n}}{\varphi_k^{\ell n}}}}.
\end{align*}
It follows from \eqref{k-Pell-Binet-eq} with replace $n$ by $\ell n$ that
\begin{align*}
 \int_{-1}^{1}(Q_{k,\ell}+2\sqrt{1+k}\, P_{k,\ell} x)^{n-1}dx
&=\frac{2^n}{nP_{k,\ell}}P_{k,\ell n}.
\end{align*}
Then  \eqref{ckPell} has been proved.
\end{proof}

The integral representations of the $k$-Pell number for even and odd orders are shown as follows:
\begin{corollary}Let $k$ and $n$ be non-negative integers with $k\neq 0$.
\begin{enumerate}[label=(\roman*)]
       \item The $k$-Pell numbers $P_{k,2n}$ can be represented by the integral
\begin{equation}\label{ckPelleven}
P_{k,2n}=n\int_{-1}^{1}(k+2+2\sqrt{1+k}\,x)^{n-1}dx.
\end{equation}
       \item The $k$-Pell numbers $P_{k,2n+1}$ can be represented by the integral
       \begin{equation*}
P_{k,2n+1}=\frac{1}{2}\int_{-1}^{1}(2n+k+2+2(n+1)\sqrt{1+k}\,x)(k+2+2\sqrt{1+k}\,x)^{n-1}dx.
\end{equation*}
     \end{enumerate}
\end{corollary}
\begin{proof}\begin{enumerate}[label=(\roman*)]
       \item  Notice that $P_{k,2}=2$ and $Q_{k,2}=2k+4$. Setting $\ell=2$ in \eqref{ckPell}, we have
\begin{align*}
P_{k,2n}&=\frac{nP_{k,2}}{2^n}\int_{-1}^{1}(Q_{k,2}+2\sqrt{1+k}\, P_{k,2} x)^{n-1}dx\\
&=\frac{n}{2^{n-1}}\int_{-1}^{1}(2k+4+4\sqrt{1+k}\,x)^{n-1}dx\\
&=n\int_{-1}^{1}(k+2+2\sqrt{1+k}\,x)^{n-1}dx.
\end{align*}
 \item Reindexing of $n$ by $n+1$ in \eqref{ckPelleven}, we get
 \begin{equation}\label{ckPelleven2}
P_{k,2n+2}=(n+1)\int_{-1}^{1}(k+2+2\sqrt{1+k}\,x)^{n}dx.
\end{equation}
Using $P_{k,2n+2}=2P_{k,2n+1}+kP_{k,2n}$ with \eqref{ckPelleven} and \eqref{ckPelleven2}, we obtain
\begin{align*}
&P_{k,2n+1}\\&=\frac{1}{2}\mat{(n+1)\int_{-1}^{1}(k+2+2\sqrt{1+k}\,x)^{n}dx
-nk\int_{-1}^{1}(k+2+2\sqrt{1+k}\,x)^{n-1}dx}\\
&=\frac{1}{2}\int_{-1}^{1}\br{(n+1)(k+2+2\sqrt{1+k}\,x)-nk}(k+2+2\sqrt{1+k}\,x)^{n-1}dx\\
&=\frac{1}{2}\int_{-1}^{1}(2n+k+2+2(n+1)\sqrt{1+k}\,x)(k+2+2\sqrt{1+k}\,x)^{n-1}dx.
\end{align*}
   \end{enumerate}
This completes the proof.
\end{proof}
Setting $k=1$ in Theorem \ref{C-k-Pell-Theorem1}, we have the following corollary.
\begin{corollary}[\cite{ANil2024}, Theorem 3.1] For $\ell$ and $n$ are non-negative integers, the Pell numbers $P_{\ell n}$ can be represented by the integral
\begin{equation*}
P_{\ell n}=\frac{nP_{\ell}}{2^n}\int_{-1}^{1}(Q_{\ell}+2\sqrt{2} P_{\ell} x)^{n-1}dx.
\end{equation*}
\end{corollary}
\begin{proof} Notice that $P_{1,\ell n}=P_{\ell n}$, $P_{1,\ell}=P_{\ell}$, and $Q_{1,\ell}=Q_\ell$. Then, by Theorem \ref{C-k-Pell-Theorem1}, the conclusion follows.
\end{proof}
Setting $k=4$ in Theorem \ref{C-k-Pell-Theorem1}, we have the following corollary.
\begin{corollary}[\cite{Ste2022}, Theorem 2.1] For $\ell$ and $n$ are non-negative integers, the Fibonacci numbers $F_{\ell n}$ can be represented by the integral
\begin{equation*}
F_{\ell n}=\frac{nF_{\ell}}{2^n}\int_{-1}^{1}(L_{\ell}+\sqrt{5} F_{\ell} x)^{n-1}dx.
\end{equation*}
\end{corollary}
\begin{proof} Notice that $P_{4,\ell n}=2^{\ell n-1}F_{\ell n}$, $P_{4,\ell}=2^{\ell -1}F_{\ell}$, and $Q_{4,\ell}=2^\ell L_\ell$. Setting $k=4$ in \eqref{ckPell}, we get
\begin{align*}
F_{\ell n}&=\frac{P_{4,\ell n}}{2^{\ell n-1}}=\frac{nP_{4,\ell}}{2^{\ell n-1+n}}\int_{-1}^{1}\br{Q_{4,\ell}+2\sqrt{1+4} P_{4,\ell} x}^{n-1}dx\\
&=\frac{n2^{\ell-1}F_\ell}{2^{\ell n-1+n}}\int_{-1}^{1}\br{2^\ell L_\ell+2^\ell\sqrt{5}F_\ell x}^{n-1}dx\\
&=\frac{nF_{\ell}}{2^n}\int_{-1}^{1}(L_{\ell}+\sqrt{5} F_{\ell} x)^{n-1}dx.
\end{align*}
This completes the proof.
\end{proof}

Next, we provide the integral representations for the $k$-Pell-Lucas numbers $Q_{k,\ell n}$ based on the two numbers $P_{k,\ell}$ and $Q_{k,\ell}$.
\begin{theorem}\label{C-k-Pell-Lucas-Theorem1}For $k$, $\ell$, and $n$ are non-negative integers with $k\neq 0$, the $k$-Pell-Lucas numbers $Q_{k,\ell n}$ can be represented by the integral
\begin{equation}\label{ckPell-Lucas}
Q_{k,\ell n}=\frac{1}{2^n}\int_{-1}^{1}(Q_{k,\ell}+2(n+1)\sqrt{1+k}\, P_{k,\ell} x)(Q_{k,\ell}+2\sqrt{1+k}\, P_{k,\ell} x)^{n-1}dx.
\end{equation}
\end{theorem}
\begin{proof}For $n=0$ or $\ell=0$, it is easy to see that \eqref{ckPell-Lucas} holds. We assume  now that $\ell, n>0$.  Replacing $n$ by $n+1$ in \eqref{ckPell} becomes
\begin{equation}\label{ckPell+}
P_{k,\ell n+\ell}=\frac{(n+1)P_{k,\ell}}{2^{n+1}}\int_{-1}^{1}(Q_{k,\ell}+2\sqrt{1+k}\, P_{k,\ell} x)^{n}dx.
\end{equation}
Integrating by part \eqref{ckPell-Lucas} and using \eqref{ckPell+}, we have
\begin{align*}
I&=\frac{1}{2^n}\int_{-1}^{1}(Q_{k,\ell}+2(n+1)\sqrt{1+k}\, P_{k,\ell} x)(Q_{k,\ell}+2\sqrt{1+k} \,P_{k,\ell} x)^{n-1}dx \\
&=\frac{1}{n2^{n+1}\sqrt{1+k}\, P_{k,\ell}}(Q_{k,\ell}+2\sqrt{1+k}\, P_{k,\ell} x)^{n}(Q_{k,\ell}+2(n+1)\sqrt{1+k}\, P_{k,\ell} x)\bigg|_{-1}^{1}\\
&\quad -\frac{n+1}{n2^n}\int_{-1}^{1}(Q_{k,\ell}+2\sqrt{1+k}\, P_{k,\ell} x)^{n}dx\\
&=\frac{1}{n2^{n+1}\sqrt{1+k}\, P_{k,\ell}}\mat{(Q_{k,\ell}+2\sqrt{1+k}\, P_{k,\ell})^{n}(Q_{k,\ell}+2(n+1) \sqrt{1+k}\,P_{k,\ell})}\\
&\quad -\frac{1}{n2^{n+1}\sqrt{1+k}\, P_{k,\ell}}\mat{(Q_{k,\ell}-2\sqrt{1+k}\, P_{k,\ell})^{n}(Q_{k,\ell}-2(n+1)\sqrt{1+k}\, P_{k,\ell})}\\
&\quad -\frac{2P_{k,\ell n+\ell}}{nP_{k,\ell}}.
\end{align*}
Applying \ref{lemma2.1-1} and  \ref{lemma2.1-2} of Lemma \ref{lemma2.1} to the righthand side of the above equation gives
\begin{align*}
I&=\frac{1}{n2^{n+1}\sqrt{1+k}\, P_{k,\ell}}\mat{2^n\varphi_k^{\ell n}\br{Q_{k,\ell}+2(n+1)\sqrt{1+k}\, P_{k,\ell}}}\\
&\quad -\frac{1}{n2^{n+1}\sqrt{1+k}\, P_{k,\ell}}\mat{2^n\frac{\br{-k}^{\ell n}}{\varphi_k^{\ell n}}\br{Q_{k,\ell}-2(n+1)\sqrt{1+k}\, P_{k,\ell}}}\\
&\quad -\frac{2P_{k,\ell n+\ell}}{nP_{k,\ell}}\\
&=\frac{1}{n\, P_{k,\ell}}\mat{\frac{1}{2\sqrt{1+k}}\br{\varphi_k^{\ell n}-\frac{\br{-k}^{\ell n}}{\varphi_k^{\ell n}}}Q_{k,\ell}+(n+1)P_{k,\ell}\br{\varphi_k^{\ell n}+\frac{\br{-k}^{\ell n}}{\varphi_k^{\ell n}}}}\\ &\quad-\frac{2P_{k,\ell n+\ell}}{nP_{k,\ell}}.
\end{align*}
Using both Binet’s formulas \eqref{k-Pell-Binet-eq} and \eqref{k-Pell-Lucas-Binet-eq} with $n$ replaced by $\ell n$ and Lemma \ref{lemma2.2}  \ref{lemma2.2-1} leads to
\begin{align*}
I&=\frac{1}{n\, P_{k,\ell}}\mat{P_{k,\ell n}Q_{k,\ell}+(n+1)P_{k,\ell}Q_{k,\ell n}}-\frac{2P_{k,\ell n+\ell}}{nP_{k,\ell}}\\
&=\frac{1}{n\, P_{k,\ell}}\mat{P_{k,\ell n}Q_{k,\ell}+P_{k,\ell}Q_{k,\ell n}+nP_{k,\ell}Q_{k,\ell n}}-\frac{2P_{k,\ell n+\ell}}{nP_{k,\ell}}\\
&=\frac{1}{n\, P_{k,\ell}}\mat{2P_{k,\ell n+\ell}+nP_{k,\ell}Q_{k,\ell n}}-\frac{2P_{k,\ell n+\ell}}{nP_{k,\ell}}\\
&=Q_{k,\ell n},
\end{align*}
which completes the proof.		
\end{proof}

Setting $k=1$ in Theorem \ref{C-k-Pell-Lucas-Theorem1}, we have the following corollary.
\begin{corollary}[\cite{ANil2024}, Theorem 3.4] For $\ell$ and $n$ are  non-negative integers, the Pell-Lucas numbers $Q_{\ell n}$ can be represented by the integral
\begin{equation*}
Q_{\ell n}=\frac{1}{2^n}\int_{-1}^{1}(Q_{\ell}+2\sqrt{2}(n+1) P_{\ell} x)(Q_{\ell}+2\sqrt{2} P_{\ell} x)^{n-1}dx.
\end{equation*}
\end{corollary}
Setting $k=4$ in Theorem \ref{C-k-Pell-Lucas-Theorem1}, we have the following corollary.
\begin{corollary}[\cite{Ste2022}, Theorem 2.2] For $\ell$ and $n$ are  non-negative integers, the Lucas numbers $L_{\ell n}$ can be represented by the integral
\begin{equation*}
L_{\ell n}=\frac{1}{2^n}\int_{-1}^{1} (L_\ell+\sqrt{5}(n+1)F_\ell x)(L_\ell+\sqrt{5}F_\ell x)^{n-1} dx.
\end{equation*}
\end{corollary}
\begin{proof} Notice that $P_{4,\ell}=2^{\ell -1}F_{\ell}$, $Q_{4,\ell n}=2^{\ell n}L_{\ell n}$, and  $Q_{4,\ell}=2^\ell L_\ell$. Setting $k=4$ in \eqref{ckPell-Lucas}, we get
\begin{align*}
L_{\ell n}&=\frac{Q_{4,\ell n}}{2^{\ell n}}\\
&=\frac{1}{2^{\ell n+n}}\int_{-1}^{1}(Q_{4,\ell}+2(n+1)\sqrt{1+4}\, P_{4,\ell} x)(Q_{4,\ell}+2\sqrt{1+4}\, P_{4,\ell} x)^{n-1}dx\\
&=\frac{1}{2^{\ell n+n}}\int_{-1}^{1}(2^\ell L_\ell+2^{\ell}(n+1)\sqrt{5}\,F_{\ell} x)(2^\ell L_\ell+2^\ell\sqrt{5}\, F_{\ell} x)^{n-1}dx\\
&=\frac{1}{2^n}\int_{-1}^{1} (L_\ell+\sqrt{5}(n+1)F_\ell x)(L_\ell+\sqrt{5}F_\ell x)^{n-1} dx.
\end{align*}
This completes the proof.
\end{proof}
Finally, both $P_{k,\ell n}$ and $Q_{k, \ell n}$ are then used to establish integral representations for the $k$-Pell numbers $P_{k,\ell n+r}$ and $k$-Pell-Lucas numbers $Q_{k, \ell n+r}$ as the following theorems.
\begin{theorem}\label{C-k-Pell-Theorem2}For $k$, $\ell$, $n$, and $r$ are  non-negative integers with $k\neq 0$, the $k$-Pell numbers $P_{k,\ell n+r}$ can be represented by the integral
$$
P_{k,\ell n+r}=\frac{1}{2^{n+1}}\hspace{-.1cm}\int_{-1}^{1}\hspace{-.1cm}\br{nP_{k,\ell} Q_{k,r}+P_{k,r} Q_{k,\ell}+(n+1)\Delta_k \,P_{k,\ell} P_{k,r} x}(Q_{k,\ell}+\Delta_k \,P_{k,\ell} x)^{n-1}dx,
$$
where $\Delta_k =2\sqrt{1+k}.$
\end{theorem}
\begin{proof}Using Lemma \ref{lemma2.2} \ref{lemma2.2-1} with $m$ and $n$ replaced by $\ell n$ and $r$ respectively, we get
$$P_{k,\ell n+r}=\frac{1}{2}P_{k,\ell n}Q_{k,r}+\frac{1}{2}P_{k,r}Q_{k,\ell n}.$$
Applying Theorems \ref{C-k-Pell-Theorem1}  and \ref{C-k-Pell-Lucas-Theorem1} leads to
\begin{align*}
&P_{k,\ell n+r}\\
&=\frac{1}{2}\br{\frac{nP_{k,\ell}}{2^n}\int_{-1}^{1}(Q_{k,\ell}+\Delta_k\,P_{k,\ell} x)^{n-1}dx}Q_{k,r}\\
&\quad+\frac{1}{2}P_{k,r}\br{\frac{1}{2^n}\int_{-1}^{1}(Q_{k,\ell}+(n+1)\Delta_k\, P_{k,\ell} x)(Q_{k,\ell}+\Delta_k\, P_{k,\ell} x)^{n-1}dx}\\
&=\frac{1}{2^{n+1}}\hspace{-.1cm}\int_{-1}^{1}\hspace{-.1cm}\br{nP_{k,\ell} Q_{k,r}+P_{k,r} Q_{k,\ell}+(n+1)\Delta_k \,P_{k,\ell} P_{k,r} x}(Q_{k,\ell}+\Delta_k \,P_{k,\ell} x)^{n-1}dx.
\end{align*}
This completes the proof.
\end{proof}
Setting $k=1$ in Theorem \ref{C-k-Pell-Theorem2}, we have the following corollary.
\begin{corollary}[\cite{ANil2024}, Theorem 3.5] For $\ell$, $n$, and $r$ are  non-negative integers, the Pell numbers $P_{\ell n+r}$ can be represented by the integral
$$
P_{\ell n+r}=\frac{1}{2^{n+1}}\int_{-1}^{1}\br{nP_{\ell} Q_{r}+P_{r} Q_{\ell}+2\sqrt{2}(n+1) \,P_{\ell} P_{r} x}(Q_{\ell}+2\sqrt{2}\,P_{\ell} x)^{n-1}dx.
$$
\end{corollary}
Setting $k=4$ in Theorem \ref{C-k-Pell-Theorem2}, we have the following corollary.
\begin{corollary}[\cite{Ste2022}, Theorem 2.3] For $\ell$, $n$, and $r$ are  non-negative integers, the Fibonacci numbers $F_{\ell n+r}$ can be represented by the integral
$$
F_{\ell n+r}=\frac{1}{2^{n+1}}\int_{-1}^{1}\br{nF_{\ell} L_{r}+F_{r} L_{\ell}+\sqrt{5}(n+1) \,F_{\ell} F_{r} x}(L_{\ell}+\sqrt{5}\,F_{\ell} x)^{n-1}dx.
$$
\end{corollary}
\begin{theorem}\label{C-k-Pell-Lucas-Theorem2}For $k$, $\ell$, $n$, and $r$ are  non-negative integers with $k\neq 0$, the $k$-Pell-Lucas numbers $Q_{k,\ell n+r}$ can be represented by the integral
$$
Q_{k,\ell n+r}\hspace{-.1cm}=\frac{1}{2^{n+1}}\hspace{-.1cm}\int_{-1}^{1}\hspace{-.1cm}\br{n\Delta_k^2P_{k,\ell} P_{k,r}\hspace{-.1cm}+Q_{k,\ell}Q_{k,r}\hspace{-.1cm}+(n+1)\Delta_k \,P_{k,\ell} Q_{k,r} x}(Q_{k,\ell}+\Delta_k \,P_{k,\ell} x)^{n-1}dx,
$$
where $\Delta_k =2\sqrt{1+k}.$
\end{theorem}
\begin{proof}
Using Lemma \ref{lemma2.2} \ref{lemma2.2-2} with $m$ and $n$ replaced by $\ell n$ and $r$ respectively, we get
$$Q_{k,\ell n+r}=\frac{1}{2}Q_{k,\ell n}Q_{k,r}+\frac{\Delta_k^2}{2}P_{k,\ell n}P_{k,r}.$$
This together with  Theorems \ref{C-k-Pell-Theorem1}  and \ref{C-k-Pell-Lucas-Theorem1} gives that  the proof is finish.
\end{proof}
Setting $k=1$ in Theorem \ref{C-k-Pell-Lucas-Theorem2}, we have the following corollary.
\begin{corollary}[\cite{ANil2024}, Theorem 3.6] For $\ell$, $n$, and $r$ are  non-negative integers, the Pell-Lucas numbers $Q_{\ell n+r}$ can be represented by the integral
$$
Q_{\ell n+r}=\frac{1}{2^{n+1}}\int_{-1}^{1}\br{8nP_{\ell} P_{r}+Q_{\ell} Q_{r}+2\sqrt{2}(n+1)  \,P_{\ell} Q_{r} x}(Q_{\ell}+2\sqrt{2}\,P_{\ell} x)^{n-1}dx.
$$
\end{corollary}
Setting $k=4$ in Theorem \ref{C-k-Pell-Lucas-Theorem2}, we have the following corollary.
\begin{corollary}[\cite{Ste2022}, Theorem 2.4] For $\ell$, $n$, and $r$ are  non-negative integers, the Lucas numbers $L_{\ell n+r}$ can be represented by the integral
$$
L_{\ell n+r}=\frac{1}{2^{n+1}}\int_{-1}^{1}\br{5nF_{\ell} F_{r}+L_{\ell} L_{r}+\sqrt{5}(n+1)  \,F_{\ell} L_{r} x}(L_{\ell}+\sqrt{5}\,F_{\ell} x)^{n-1}dx.
$$
\end{corollary}
\section{Conclusions}
In this paper, new integral representations of the $k$-Pell  and $k$-Pell-Lucas numbers have been introduced and studied. Many of the properties of these numbers are proved by using Binet’s formulas. We also establish some identities and simple integral calculus to prove them. Moreover, our results are also deduced and related to the Fibonacci, Lucas, Pell, and  Pell-Lucas numbers.
\section*{Acknowledgements}
We would like to thank the referees for their comments and suggestions on the manuscript. The first author is supported by the Department of Mathematics, Faculty of Science, Ubon Ratchatani Rajabhat University.

\label{LastPage}

\begin{thebibliography}{20}
\bibitem{Bic1915} M. Bicknell, A primer on the Pell sequence and related sequences, Fibonacci Quart. 13 (4) (1915)  345--349.

\vskip0.1cm

\bibitem{Hor1971} A. F. Horadam, Pell identities, Fibonacci Quart. 9 (3) (1971) 245--263.

\vskip0.1cm


\bibitem{HorMah1985} A. F. Horadam, J. M. Mahon, Pell and Pell-Lucas polynomials, Fibonacci Quart. 23 (1) (1985) 7--20.

    \vskip0.1cm

\bibitem{Dil2000}K. Dilcher, Hypergeometric functions and Fibonacci numbers, Fibonacci Quart. 38 (4) (2000) 342--363.

\vskip0.1cm

\bibitem{Vaj2008}S. Vajda, Fibonacci and Lucas Numbers, and the Golden Section: Theory and Applications, Dover Press, 2008.

\vskip0.1cm

\bibitem{GlaZho2015}M. L. Glasser, Y. Zhou, An integral representation for the Fibonacci numbers and their generalization, Fibonacci Quart. 53 (4) (2015) 313--318.

\vskip0.1cm

\bibitem{Kos2018}T. Koshy, Fibonacci and Lucas Numbers with Applications, 2nd, NJ: John Wiley \& Sons, 2018.

\vskip0.1cm

\bibitem{AndBag2020} D. Andrica, O. Bagdasar, Recurrent Sequences: Key Results, Applications, and Problems, Springer, Cham, 2020.

 \vskip0.1cm

\bibitem{ErdKes2022}F. Erduvan, R. Keskin, Pell and Pell-Lucas numbers as product of two repdigits, Mathematical Notes 112 (6) (2022) 861--871.

\vskip0.1cm

\bibitem{Ste2022}S. M. Stewart, Simple integral representations for the Fibonacci and Lucas numbers, Aust. J. Math. Anal. Appl. 19 (2) (2022) 1--5.

    \vskip0.1cm

\bibitem{Ste2023}S. M. Stewart, A simple integral representation of the Fibonacci numbers, Mathematical Gazette  107 (568) (2023) 120--123.
   \vskip0.1cm

\bibitem{ANil2024} A. Nilsrakoo, Integral representations of the Pell and Pell-Lucas numbers, J. Science and Science Education 7 (2) (2024) 272--281.

  \vskip0.1cm

\bibitem{Cat2013}P. Catarino, On some identities and generating functions for $k$-Pell numbers, Int. J. Math. Anal. (Ruse) 7 (2013)  1877--1884.

\vskip0.1cm

\bibitem{CatVas2013}P. Catarino, P. Vasco, On some identities and generating functions for $k$-Pell-Lucas sequence, Appl. Math. Sci. (Ruse) 7 (2013) 4867--4873.

\vskip0.1cm

\bibitem{OEIS2024}OEIS Foundation Inc,  The on-line encyclopedia of integer
sequences, http://oeis.org, 2024.

\end{thebibliography}
\end{document}